\documentclass[3p]{elsarticle}

\usepackage{amsmath,amsthm,amssymb,amsfonts}
\usepackage[english]{babel} 

\usepackage{times}

\newtheorem{thm}{Theorem}[section]
\newtheorem{lem}[thm]{Lemma}
\newtheorem{cor}[thm]{Corollary}
\newtheorem{prop}[thm]{Proposition}

\newtheorem{rem}[thm]{Remark}

\DeclareMathAlphabet{\mathpzc}{OT1}{pzc}{m}{it}

\numberwithin{equation}{section}

\newcommand{\Wqb}{W_{q,D}}
\newcommand{\Wqq}{\Wqb^{2-2/q}}
\newcommand{\Wqqp}{\Wqb^{2-2/q,+}}
\newcommand{\R}{\mathbb{R}}

\newcommand{\N}{\mathbb{N}}
\newcommand{\A}{\mathbb{A}}

\newcommand{\X}{\mathbb{X}}
\newcommand{\Wq}{\mathbb{W}_q}
\newcommand{\Wqh}{\hat{\mathbb{W}}_q}
\newcommand{\Wqd}{\dot{\mathbb{W}}_q^+}
\newcommand{\Lq}{\mathbb{L}_q}

\newcommand{\ml}{\mathcal{L}}
\newcommand{\mk}{\mathcal{K}}

\newcommand{\Om}{\Omega}

\newcommand{\ve}{\varepsilon}
\newcommand{\rd}{\mathrm{d}}
\newcommand{\divv}{\mathrm{div}}

\newcommand{\bqn}{\begin{equation}}
\newcommand{\eqn}{\end{equation}}
\newcommand{\bqnn}{\begin{equation*}}
\newcommand{\eqnn}{\end{equation*}}
\newcommand{\bear}{\begin{eqnarray}} 
\newcommand{\eear}{\end{eqnarray}} 
\newcommand{\bean}{\begin{eqnarray*}} 
\newcommand{\eean}{\end{eqnarray*}} 
\newcommand{\bs}{\begin{split}}
\newcommand{\es}{\end{split}}

\newcommand{\dimens}{\mathrm{dim}}
\newcommand{\codim}{\mathrm{codim}}
\newcommand{\kk}{\mathrm{ker}}
\newcommand{\im}{\mathrm{rg}}

\newcommand{\dhr}{\mathrel{\lhook\joinrel\relbar\kern-.8ex\joinrel\lhook\joinrel\rightarrow}} 

\begin{document}

\title{
Positive solutions of some parabolic system with cross-diffusion and nonlocal initial conditions}

\author[addrHannover]{Christoph Walker}
\ead{walker@ifam.uni-hannover.de}
\address[addrHannover]{Leibniz Universit\"at Hannover, Institut f\" ur Angewandte Mathematik, \\ Welfengarten 1, D--30167 Hannover, Germany}

\begin{abstract}
The paper is concerned with a system consisting of two coupled nonlinear parabolic equations with a cross-diffusion term, where the solutions at positive times define the initial states. The equations arise as steady state equations of an age-structured predator-prey system with spatial dispersion. Based
on unilateral global bifurcation methods for Fredholm operators and on maximal regularity for parabolic equations, global bifurcation of positive solutions is derived.
\end{abstract}


\maketitle

\section{Introduction}

Consider the situation that an age-structured prey and an age-structured predator population inhabit the same spatial region $\Om$, that the individuals of both populations undergo spatial fluctuation and that the predator population exerts a repulsive pressure on the prey population. If $u=u(t,a,x)\ge 0$ and $v=v(t,a,x)\ge 0$ denote the  density of the prey and the predator population, respectively, at time $t\ge 0$, age $a\in [0,a_m)$ for a maximal age $a_m>0$, and spatial position $x\in\Om$, a simple model reads
\begin{align}
\partial_t u+\partial_a u-\Delta_x\left(\left(\delta_1+\gamma v\right)u\right) &=-\alpha_1u^2-\alpha_2 uv\ , \quad t>0\ ,\quad a\in (0,a_m)\ ,\quad x\in\Om\ ,  \label{1} \\
\partial_t v+\partial_a v-\delta_2\Delta_x v&=-\beta_1 v^2+\beta_2 vu\ , \quad t>0\ ,\quad a\in (0,a_m)\ ,\quad x\in\Om\ .\label{2} 
\end{align}
These equations are subject to the nonlocal age boundary conditions
\begin{align}
u(t,0,x)&=\int_0^{a_m}B_1(a)\, u(t,a,x)\, \rd a\ , \quad t>0\ ,\quad  x\in\Om\ ,\label{3}\\
v(t,0,x)&=\int_0^{a_m}B_2(a)\, v(t,a,x)\, \rd a\ ,  \quad t>0\ ,\quad x\in\Om\ ,\label{4}
\end{align}
the spatial boundary conditions
\begin{align}
u(t,a,x)&=0\ , \quad  t>0\ ,\quad a\in (0,a_m)\ ,\quad x\in\partial\Om\ ,\label{5}\\
v(t,a,x)&=0\ , \quad  t>0\ ,\quad a\in(0,a_m)\ ,\quad x\in\partial\Om\ ,\label{6}
\end{align}
and are supplemented with time initial conditions.
The $\Delta_x$-term in \eqref{1} describes spatial movement of prey individuals. Besides intrinsic dispersion with coefficient $\delta_1>0$, it reflects an increase of the dispersive force on the prey by repulsive interference with an increase of the predator population. Here, $\gamma\ge 0$ is the predator population pressure coefficient. We refer to \cite{ShigesadaEtAl} for a derivation of such kind of models (without age-structure). The right hand sides of \eqref{1} and \eqref{2} take into account intra- and inter-specific interactions of the two populations with positive coefficients $\alpha_1, \alpha_2, \beta_1, \beta_2>0$. Equations \eqref{5} and \eqref{6} describe creation of new individuals 
with nonnegative birth rates $B_1$ and $B_2$. 
The reader is referred to \cite{WebbSpringer} and the references therein for further information about linear and nonlinear age-structured population equations with (and without) spatial dispersal. Although such models have a long history, there does not seem to be much literature about equations with nonlinear diffusion. The well-posedness of equations \eqref{1}-\eqref{6} might be derived within the general framework presented in \cite{WalkerDCDS} though the results therein do not directly apply due to the cross-diffusion term.

We shall remark that the model considered herein is a rather simple biological model and there may be more accurate models, e.g. with different maximal ages, nonlocal dependences in the reaction terms etc. The aim of the present paper is thus rather to provide a mathematical framework to treat parabolic equations with cross-diffusion and nonlocal initial conditions. Namely, the present paper is dedicated to positive equilibrium (i.e. time-independent) solutions to \eqref{1}-\eqref{6} which shall be established based on global bifurcation methods. We write the birth rates in the form $B_1(a)=\eta b_1(a)$ and $B_2(a)=\xi b_2(a)$, where  $b_1$, $b_2$ are some fixed birth profiles and the parameters $\eta$, $\xi$ measuring the intensities of the fertility shall serve as bifurcation parameters. 
Aiming at a simple and compact notation, we let $b:=b_1=b_2$ and $\delta_1=\delta_2=1$. 
We emphasize that these simplifications are made merely for the sake of readability and do not impact in any way on the mathematical analysis to follow. To shorten notation further, we shall agree upon the following convention: given a function defined on $J:=[0,a_m]$ and denoted by a lower case letter, say $u$ or $v$, we use the corresponding capital letter $U$ or $V$ to denote its age-integral with weight $b$, i.e.,
\bqn\label{7}
U:=\int_0^{a_m} b(a) u(a)\,\rd a\ ,\quad V:=\int_0^{a_m} b(a) v(a)\,\rd a\ .
\eqn
Depending on the values of the parameters $\eta$ and $\xi$, we are looking for nonnegative and nontrivial functions $u=u(a,x)$ and $v=v(a,x)$ satisfying the nonlinear parabolic system with cross-diffusion
\begin{align}
\partial_a u-\Delta_x\big((1+\gamma v)u\big) &=-\alpha_1u^2-\alpha_2 uv\ ,\quad  a\in (0,a_m)\ ,\quad x\in\Om\ ,\label{1a}\\
\partial_a v-\Delta_x v&=-\beta_1 v^2+\beta_2 vu\ ,\,\quad  a\in (0,a_m)\ ,\quad x\in\Om\ ,\label{2a}
\end{align}
subject to the nonlocal initial conditions
\begin{align}
u(0,x)&=\eta U\ , \quad   x\in\Om\ ,\label{3a}\\
v(0,x)&=\xi V\ ,  \quad x\in\Om\ ,\label{4a}
\end{align}
and spatial Dirichlet boundary conditions
\begin{align}
u(a,x)&=0\ , \quad  a\in (0,a_m)\ ,\quad x\in\partial\Om\ ,\label{5a}\\
v(a,x)&=0\ , \quad  a\in(0,a_m)\ ,\quad x\in\partial\Om\ .\label{6a}
\end{align}
Clearly, of particular interest are {\it coexistence states}, that is, solutions $(u,v)$ with both components $u$ and $v$ nontrivial and nonnegative. 

Based on bifurcation techniques, existence of equilibrium solutions was established in \cite{WalkerJRAM,WalkerJFA} for similar systems with linear diffusion and in \cite{WalkerSIMA, WalkerJDE,WalkerAMPA} for a single equation with nonlinear diffusion. Except for \cite{WalkerAMPA}, the bifurcation parameter was also chosen to be a measure for the intensity of the fertility as above. Prior to the just cited papers, bifurcation methods were used in \cite{DelgadoEtAl2} to derive positive equilibrium solutions for a single equations with linear diffusion. We also refer to \cite{Langlais} where the large time behavior of population dynamics with age-dependence and linear spatial diffusion was analyzed. 

More attention attracted than the age-structured parabolic equations so far have related elliptic systems of the form 
$$
 -\Delta_x [(1+\gamma v)u]=u(\mu -u\pm cv)\ ,\quad -\Delta_x v=v(\lambda-v\pm bu)
$$
in $\Om$ subject to Dirichlet conditions on $\partial\Om$ with positive constants $\mu, \lambda, b, c$; 
both for the case of linear diffusion $\gamma=0$ (e.g., see \cite{BlatBrown1,BlatBrown2,DelgadoLGSuarez,Lou,Pao} and the references therein) and the case with cross-diffusion $\gamma>0$ (e.g., see \cite{DelgadoMontenegroSuarez,Horstmann,Kuto1,KutoYamada,ShiWang} and the references therein). It is worthwhile to remark that for such elliptic equations with cross-diffusion, the transformation $z:=(1+\gamma v)u$ leads to a semilinear elliptic system for $z$ and $v$, for which, when written in the form $(z,v)=S(z,v)$, the corresponding solution operator $S$ enjoys suitable compactness properties. Moreover, for such a system, positivity and a-priori bounds of solutions may be derived using the maximum principle, e.g. \cite{DelgadoMontenegroSuarez,Kuto1,ShiWang}. For the nonlocal parabolic equations \eqref{7}-\eqref{6a} under consideration herein, however, a corresponding transformation still yields nonlinear second order terms with respect to the spatial variable (or first order time derivatives). Consequently, the underlying solution operator does not enjoy similar compactness properties. As we are interested in global bifurcation of positive solutions, it is thus not clear how to apply directly the unilateral global bifurcation methods of \cite{LopezGomezChapman,Rabinowitz} as in \cite{WalkerJRAM}, which require compactness of the underlying solution operators. To overcome this deficiency, we shall invoke recent results of Shi $\&$ Wang \cite{ShiWang} on unilateral global bifurcation for Fredholm operators, which are based on L\'opez-G\'omez's interpretation \cite{LopezGomezChapman} of the global alternative of Rabinowitz \cite{Rabinowitz} and on the global bifurcation results of Pejsachowicz $\&$ Rabier \cite{PejsachowiczRabier}. These results yield a (global) continuum of positive solutions. 
We shall also point out that, due to the cross-diffusion term, no comparison principle in the spirit of \cite[Lem.3.2]{WalkerJRAM} is available. In some cases, the lack of such a comparison principle prevents us from determining which of the possible alternatives the constructed continuum of coexistence states satisfies (see Theorem~\ref{T22} and Theorem~\ref{T222} below).

\section{Main Results}\label{main results}

\noindent We suppose throughout this paper that the birth profile
\bqn\label{10}
b\in L_\infty^+((0,a_m))\ \text{with}\ b(a)>0\ \text{for $a$ near $a_m$}
\eqn
is normalized such that
\bqn\label{6n}
\int_0^{a_m}b(a)e^{-\lambda_1 a}\,\rd a=1\ ,
\eqn
where $\lambda_1>0$ denotes the principal eigenvalue of $-\Delta_x$ on $\Om$ subject to Dirichlet boundary conditions on~$\partial\Om$. Before stating our results on coexistence states in more detail, we first recall some auxiliary results from \cite{WalkerJRAM} about semi-trivial states, that is, about solutions $(u,v)$ with one vanishing component. Taking for instance $v\equiv 0$ in \eqref{1a}-\eqref{2a} and using convention \eqref{7}, we obtain the reduced problem
\bqn\label{Aee1}
\partial_a u-\Delta_x u=-\alpha_1u^2\ ,\quad u(0,\cdot)=\eta U\ ,
\eqn
subject to Dirichlet boundary conditions. 
Introducing, for $q\in (n+2,\infty)$ fixed, the solution space 
$$
\Wq:=L_q((0,a_m),\Wqb^2(\Om))\cap W_q^1((0,a_m),L_q(\Om))\ ,
$$
where $$\Wqb^2(\Om):=\{u\in W_q^2(\Om)\,;\, u=0\, \text{on}\, \partial\Om\}\ ,$$ and letting $\Wq^+$ denote its positive cone, the following result was proved in \cite[Thm.2.1, Cor.3.3, Lem.3.6, Cor.3.7]{WalkerJRAM}:

\begin{thm}\label{A1}
Suppose \eqref{10} and \eqref{6n}.
For each $\eta>1$ there is a unique solution $u_\eta\in\Wq^+\setminus\{0\}$ to equation \eqref{Aee1}. The mapping $(\eta\mapsto u_\eta)$ belongs to $C^\infty((1,\infty),\Wq)$  and satisfies $\|u_\eta\|_{\Wq}\rightarrow 0$ as $\eta\rightarrow 1$ and $\|u_\eta\|_{\Wq}\rightarrow\infty$ as $\eta\rightarrow\infty$. For $\eta>1$ and $a\in J$, the derivative $\partial_\eta u_\eta (a)$ is a (strictly) positive function on $\Om$.
If $\eta_1 > \eta_2$, then $u_{\eta_1}\ge u_{\eta_2}$. Finally, if $\eta\le 1$, then \eqref{Aee1} has no solution in $\Wq^+\setminus\{0\}$.
\end{thm}

\noindent Given $\xi>1$, we let $v_\xi\in\Wq^+\setminus\{0\}$ denote the unique solution to the corresponding equation for $v$ when taking $u\equiv 0$ in \eqref{1a}-\eqref{2a}. Though we shall work in the solution space $\Wq$, we remark that solutions to \eqref{7}-\eqref{6a}, so in particular $u_\eta$ and $v_\xi$, are smooth with respect to both $a\in J$ and $x\in\Om$ (see Lemma~\ref{A44} below).

\subsection{Bifurcation with Respect to the Parameter $\eta$}

We first shall keep $\xi$ fixed and regard $\eta$ as bifurcation parameter. We thus write $(\eta,u,v)$ for solutions to \eqref{7}-\eqref{6a} with $u$, $v$ belonging to $\Wq^+$. Theorem~\ref{A1} entails, for any $\xi\ge 0$, the semi-trivial branch
\bqn\label{k1}
\mathfrak{B}_1:=\{(\eta,u_\eta,0)\,;\, \eta> 1\}\subset\R^+\times(\Wq^+\setminus\{0\})\times\Wq^+
\eqn
of solutions. For $\xi>1$, an additional semi-trivial branch
\bqn\label{k2}
\mathfrak{B}_2:=\{(\eta,0,v_\xi)\,;\, \eta\ge 0\}\subset\R^+\times\Wq^+\times(\Wq^+\setminus\{0\})\ 
\eqn
exists. In the latter case, a continuum of coexistence states bifurcates from $\mathfrak{B}_2$:

\begin{thm}\label{T1}
Suppose \eqref{10}-\eqref{6n} and let $\xi>1$. There exists $\eta_0:=\eta_0(\xi)>0$ such that an unbounded continuum $\mathfrak{C}$ emanates from $(\eta_0,0,v_\xi)\in\mathfrak{B}_2$, and
$$
\mathfrak{C}\setminus\{(\eta_0,0,v_\xi)\}\subset\R^+\times(\Wq^+\setminus\{0\})\times (\Wq^+\setminus\{0\})
$$
consists of coexistence states $(\eta,u,v)$ to \eqref{7}-\eqref{6a}. 
Near the bifurcation point $(\eta_0,0,v_\xi)$, $\mathfrak{C}$ is a continuous curve. There is no other bifurcation point on $\mathfrak{B}_2$ or on $\mathfrak{B}_1$ to positive coexistence states.
\end{thm}

\noindent The value of $\eta_0(\xi)$ corresponding to the bifurcation point is determined in \eqref{eta0}. 

We turn to the case $\xi<1$, which is more involved. Recall that $\mathfrak{B}_1$ is the only semi-trivial branch of solutions in this case.

\begin{thm}\label{T22}
Suppose \eqref{10}-\eqref{6n}. There is $\delta\in [0,1)$ with the property that, given $\xi\in (\delta,1)$, there exists \mbox{$\eta_1:=\eta_1(\xi)>1$} such that a continuum $\mathfrak{S}$ emanates from $(\eta_1,u_{\eta_1},0)\in\mathfrak{B}_1$, and
$$
\mathfrak{S}\setminus\{(\eta_1,u_{\eta_1},0)\}\subset \R^+\times(\Wq^+\setminus\{0\})\times (\Wq^+\setminus\{0\})
$$
consists of coexistence states $(\eta,u,v)$ of \eqref{7}-\eqref{6a}. The continuum $\mathfrak{S}$
\begin{itemize}
\item[(i)] is unbounded, or
\item[(ii)] connects $(\eta_1,u_{\eta_1},0)$ to a solution $(\eta_*,u,v)$ of \eqref{7}-\eqref{6a} with $\eta_*\in (0,1)$ and $u,v\in\Wq^+\setminus\{0\}$.
\end{itemize}
Near the bifurcation point $(\eta_1,u_{\eta_1},0)$, $\mathfrak{S}$ is a continuous curve. There is no other bifurcation point on $\mathfrak{B}_1$ to positive coexistence states.
\end{thm}

The values of $\delta$ and $\eta_1(\xi)$ are given in \eqref{delta} and \eqref{eta1}, respectively. Note that alternative (i) above always occurs if cross-diffusion is not taken into account, i.e. if $\gamma=0$, see \cite[Thm.1.5]{WalkerJFA}. Alternative (ii) stems from a technical condition, and we conjecture that alternative (i) is the generic case also for the present situation where $\gamma>0$. However, as occurrence of alternative (ii) is not ruled out by our analysis, Theorem~\ref{T22} is rather a local bifurcation result in this regard.

\subsection{Bifurcation with Respect to the Parameter $\xi$}

Since the cross-diffusion term involves the predator density $v$, bifurcation from semi-trivial solution branches is more intricate when regarding $\xi$ (the measure of the predator fertility) as parameter and keeping $\eta$ fixed. We shall only consider the situation $\eta>1$, the reason being explained in Remark~\ref{R1}. We now write $(\xi,u,v)$ for solutions to \eqref{7}-\eqref{6a}. Then, there are two semi-trivial branches of solutions
$$
\mathfrak{T}_1:=\{(\xi,0,v_\xi)\,;\, \xi> 1\}\ ,\qquad
\mathfrak{T}_2:=\{(\xi,u_\eta,0)\,;\, \xi\ge 0\}\ .
$$
Also in this case, bifurcation from the semi-trivial branch $\mathfrak{T}_2$ occurs:

\begin{thm}\label{T222}
Suppose \eqref{10}-\eqref{6n} and let $\eta>1$. There exists $\xi_0:=\xi_0(\eta)\in (0,1)$ such that a continuum $\mathfrak{R}$ of solutions to \eqref{7}-\eqref{6a} emanates from $(\xi_0,u_{\eta},0)\in\mathfrak{T}_2$ satisfying the alternatives
\begin{itemize}
\item[(i)] $\mathfrak{R}\setminus\{(\xi_0,u_{\eta},0)\}$ is unbounded in $\R^+\times (\Wq^+\setminus\{0\})\times (\Wq^+\setminus\{0\})$, or
\item[(ii)] $\mathfrak{R}$ connects $\mathfrak{T}_2$ with $\mathfrak{T}_1$.
\end{itemize}
Except for the bifurcation point(s), $\mathfrak{R}$ consists exclusively of coexistence states. 
Near the bifurcation point $(\xi_0,u_{\eta},0)$, $\mathfrak{R}$ is a continuous curve. There is no other bifurcation point on $\mathfrak{T}_2$ or on $\mathfrak{T}_1$ to positive coexistence states.
\end{thm}

The values of $\xi_0(\eta)$ and $\xi_1=\xi_1(\eta)>1$, corresponding to the connection point $(\xi_1,0,v_{\xi_1})\in\mathfrak{T}_1$ if alternative (ii) occurs, are determined in \eqref{xi0} and \eqref{xi1}, respectively. Given the assumptions of this theorem, if $\gamma=0$, one can show that alternative (ii) occurs if $\eta$ is less than a certain value or if additional assumptions on the birth rates and on $\alpha_j$, $\beta_j$ are imposed, see \cite[Thm.2.2]{WalkerJRAM}.\\

The outline of the paper is as follows: The next section is dedicated to the proof of Theorem~\ref{T1}. In Subsection~\ref{sec 2.1}, we first introduce some notation, and in Subsection~\ref{sec 2.2}, we provide the necessary auxiliary results needed to perform the actual proof of Theorem~\ref{T1} in Subsection~\ref{sec 2.3}. Sections \ref{sec 3} and \ref{sec 4} are dedicated to the proofs of Theorem~\ref{T22} and Theorem~\ref{T222}, respectively. As these proofs are along the lines of the proof of Theorem~\ref{T1}, these sections are kept rather short.

\section{Proof of Theorem~\ref{T1}}\label{sec 2}

\subsection{Notations and Preliminaries}\label{sec 2.1}

Given two Banach spaces $F_1$ and $F_0$, we shall use the notation $\mathcal{L}(F_1,F_0)$ for the set of bounded linear operators and $\mathcal{K}(F_1,F_0)$ for the set of compact linear operators between $F_1$ and $F_0$. We set $\mathcal{L}(F_0):=\mathcal{L}(F_0,F_0)$ and similarly \mbox{$\mathcal{K}(F_0):=\mathcal{K}(F_0,F_0)$}. The set of toplinear isomorphisms between $F_1$ and $F_0$ is denoted by $\mathcal{L}is(F_1,F_0)$.

Let $-\Delta_D$ denote the negative Laplacian on $L_q:=L_q(\Om)$ subject to Dirichlet boundary conditions, that is, $-\Delta_D:=-\Delta_x$ with domain $\Wqb^2$, where
$$
 \Wqb^\kappa:=\Wqb^\kappa (\Om):=\{u\in W_q^\kappa\,;\, u=0\ \text{on}\ \partial\Om\}
 $$
for $\kappa>1/q$ and $\Wqb^\kappa:=W_q^\kappa (\Om)$ for $0\le\kappa<1/q$. As $q>n+2$, we have $\Wqb^{2-2/q}\hookrightarrow C^1(\bar{\Om})$ by the Sobolev embedding theorem. In particular, $\mathrm{int}(\Wqqp)\not=\emptyset$, that is, the interior of the positive cone of $\Wqb^{2-2/q}$ is not empty.
Let $\gamma_0$ denote the trace operator defined by $\gamma_0 u:=u(0)$ for $u\in\Wq$, which is well-defined owing to the embedding  \cite[III.Thm.4.10.2]{LQPP}
$$
\Wq\hookrightarrow C\big([0,a_m],\Wqb^{2-2/q}\big)\ .
$$
In fact, we have, due to the interpolation inequality
\cite[I.Thm.2.11.1]{LQPP}, 
\bqn\label{emb}
\Wq\hookrightarrow C^{1-1/q-\vartheta}([0,a_m],\Wqb^{2\vartheta})\ ,\quad 0\le \vartheta\le 1-1/q\ .
\eqn
In the following, we let $\Wqd:=\Wq^+\setminus\{0\}$.
Recall that an operator $A\in \mathcal{L}(\Wqb^2,L_q)$ is said to have {\it maximal $L_q$-regularity} provided 
$$
(\partial_a+A,\gamma_0)\in \mathcal{L}is(\Wq,\Lq\times \Wqb^{2-2/q})\ ,
$$
where $\Lq:=L_q(J,L_q)$. Recall that $-\Delta_D$ has maximal $L_q$-regularity. If $A: J\rightarrow \mathcal{L}(\Wqb^\kappa,L_q)$ for some $\kappa\ge 0$, we write $(Au)(a):=A(a) u(a)$ for $a\in J$ and \mbox{$u:J\rightarrow \Wqb^\kappa$}. The following perturbation result is proved in \cite[Cor.3.4]{PruessMonopoli}, which is also valid for $\R^k$-valued functions, i.e., if $L_q$ is replaced by $L_q(\Om,\R^k)$ etc.:

\begin{lem}\label{Jan}
Let $A\in C(J,\mathcal{L}(\Wqb^2,L_q))$ be such that $A(a)$ has maximal $L_q$-regularity for each $a\in J$ and let $B\in L_q(J,\mathcal{L}(\Wqb^{2-2/q},L_q))$. Then $(\partial_a+A+B,\gamma_0)\in\mathcal{L}is(\Wq,\Lq\times\Wqb^{2-2/q})$.
\end{lem}

Given $\varrho>0$ and $h\in C^\varrho (J,C(\bar{\Om}))$, we let
$\Pi_{[h]}(a,\sigma)$, $0\le \sigma\le a\le a_m$,
denote the unique parabolic evolution operator corresponding to  
$-\Delta_D + h\in  C^\varrho \big(J,\ml(\Wqb^2,L_q)\big)$,
that is, $z(a)=\Pi_{[h]}(a,\sigma)\Phi$, $a\in (\sigma,a_m)$, defines the unique strong solution to 
$$
\partial_a z-\Delta_D z+hz=0\ ,\quad a\in [\sigma,a_m)\ ,\qquad z(\sigma)=\Phi\ ,
$$
for any given $\sigma\in (0,a_m)$ and $\Phi\in L_q$ (see \cite[II.Cor.4.4.1]{LQPP}). Note that the evolution operator is positive, i.e. 
$$
\Pi_{[h]}(a,\sigma)\Phi \in L_q^+\ ,\quad 0\le \sigma\le a\le a_m\ ,\quad \Phi\in L_q^+\ .
$$ 
Since $-\Delta_D$ has maximal $L_q$-regularity, it follows from Lemma~\ref{Jan} that 
\bqn\label{maxreg}
(\partial_a-\Delta_D +h,\gamma_0)\in\mathcal{L}is(\Wq,\Lq\times\Wqb^{2-2/q})
\eqn 
and, in particular, $\Pi_{[h]}(\cdot,0)\Phi\in\Wq$ for $\Phi\in\Wqb^{2-2/q}$.
We set
$$
H_{[h]}:=\int_0^{a_m}b(a)\,\Pi_{[h]}(a,0)\,\rd a\ .
$$
Then $H_{[h]}\in \mk(\Wqb^{2-2/q})$ owing to standard regularizing effects of the parabolic evolution operator $\Pi_{[h]}$ and the compact embedding $\Wqb^2\hookrightarrow\Wqb^{2-2/q}$. Moreover, as $\Pi_{[h]}(a,0)$ for $a\in (0,a_m)$ is strongly positive on $\Wqb^{2-2/q}$ by \cite[Cor.13.6]{DanersKochMedina}, the same holds true for $H_{[h]}$ due to \eqref{10}, that is,
\bqn\label{ssss}
H_{[h]}\Phi \in \mathrm{int}(\Wqqp)\ ,\quad \Phi\in\Wqqp\setminus\{0\}\ .
\eqn
The corresponding spectral radius $r(H_{[h]})$ can thus be characterized according to the Krein-Rutman theorem \cite[Thm.3.2]{AmannSIAMReview} (see \cite[Lem.3.1]{WalkerJRAM}):

\begin{lem}\label{A2}
For $h\in C^\varrho (J,C(\bar{\Om}))$ with $\varrho>0$, the spectral radius $r(H_{[h]})>0$ is a simple eigenvalue of $H_{[h]}$ with a corresponding eigenfunction belonging to $\mathrm{int}(\Wqqp)$. It is the only eigenvalue of $H_{[h]}$ with a positive eigenfunction. Moreover, if $h$ and $g$ both belong to $C^\varrho (J,C(\bar{\Om}))$ with $g\ge h$ but $g\not\equiv h$, then $r(H_{[g]})<r(H_{[h]})$. 
\end{lem}

\noindent In particular, the normalization \eqref{6n} implies
\bqn\label{1mio}
r(H_{[0]})=1\ 
\eqn
since any positive eigenfunction of $-\Delta_D$ is an eigenfunction of $H_{[h]}$ as well. Moreover, writing the solution to \eqref{Aee1} in the form $u=\Pi_{[\alpha_1 u]}(\cdot,0)u(0)$, Theorem~\ref{A1} together with Lemma~\ref{A2} imply
\bqn\label{sp}
\eta\, r(H_{[\alpha_1 u_\eta]})=\xi\, r(H_{[\beta_1 v_\xi]})=1\ ,\quad \eta,\xi>1\ ,
\eqn
since $u_\eta(0), v_\xi(0)\in(\Wqqp)\setminus\{0\}$.

\subsection{Auxiliary Results}\label{sec 2.2}

The aim is to apply the global bifurcation results of \cite[Thm.4.3,Thm.4.4]{ShiWang} in order to establish Theorem~\ref{T1}. We first provide the necessary tools.

Let $\xi>1$ be fixed and let $v_\xi\in\Wqd$ denote the solution to \eqref{7}-\eqref{6a} with $u\equiv 0$ provided by Theorem~\ref{A1}. Throughout we use the convention \eqref{7}. Notice that for each $a \in J$, the operator $A_1(a)$, given by
$$
A_1(a)u:=-\divv_x((1+\gamma v_\xi(a))\nabla_x u)\ ,\quad u\in \Wqb^2\ \ ,
$$
has maximal $L_q$-regularity due to its divergence form, the positivity of $v_\xi\in\Wq$, \eqref{emb}, and e.g. \cite[I.Cor.1.3.2, III.Ex.4.7.3, III.Thm.4.10.7]{LQPP}. Moreover, $A_1\in C(J,\mathcal{L}(\Wqb^2,L_q))$ by \eqref{emb}. Noticing also that
$$
B_1(a)u:=-\gamma\,\divv_x\big(u\nabla_x v_\xi(a)\big)+\alpha_2 v_\xi(a) u\ ,\quad a\in J\ ,\quad u\in \Wqb^{2-2/q}\ ,
$$
defines an operator $B_1\in L_q(J,\mathcal{L}(\Wqb^{2-2/q},L_q))$,
it follows from Lemma~\ref{Jan} that 
\bqn\label{T11}
T_1:=(\partial_a+A_1+B_1,\gamma_0)^{-1}\in\mathcal{L}(\Lq\times \Wqb^{2-2/q},\Wq)
\eqn
is well-defined. Observe that
$$
A_\xi u:=(A_1+B_1)u=-\Delta_D \big((1+\gamma v_\xi)u\big)+ \alpha_2 v_\xi u\ ,\quad u\in \Wqb^2\ .
$$
From \eqref{emb} and  \eqref{maxreg} we obtain
$$
T_2:=(\partial_a-\Delta_D+2\beta_1 v_\xi,\gamma_0)^{-1}\in\mathcal{L}(\Lq\times \Wqb^{2-2/q},\Wq)\ .
$$
To derive a bifurcation from a point $(\eta_0, 0,v_\xi)\in\mathfrak{B}_2$ for a suitable $\eta_0=\eta_0(\xi)$, we write $(\eta,u,v)=(\eta,u,v_\xi+w)$ for a solution to problem \eqref{7}-\eqref{6a}, which is then equivalent to
\begin{align}
\partial_a u-\Delta_D\big((1+\gamma v_\xi)u\big) &=\Delta_D\big(\gamma w u\big)-\alpha_1u^2-\alpha_2 u(v_\xi +w)\ ,  &a\in (0,a_m)\ ,\quad x\in\Om\ ,\label{1ab}\\
\partial_a w-\Delta_D w&=-\beta_1 w^2-2\beta_1v_\xi w+\beta_2 (v_\xi +w)u\ ,\,&  a\in (0,a_m)\ ,\quad x\in\Om\ ,\label{2ab}
\end{align}
subject to 
\begin{align}
u(0,x)&=\eta U(x)\ , \quad   x\in\Om\ ,\label{3ab}\\
w(0,x)&=\xi W(x)\ ,  \quad x\in\Om\ .\label{4ab}
\end{align}
The solutions $(\eta,u,w)$ to \eqref{1ab}-\eqref{4ab}, in turn, are the zeros of the map $F:\R\times\Wq\times\Wqh\rightarrow \Wq\times\Wq$, defined by
\bqn\label{F}
F(\eta,u,w):=\left(\begin{matrix} u-T_1\big(\Delta_D(\gamma wu)-\alpha_1u^2- \alpha_2 uw\, , \,\eta U\big)\\ w-T_2\big(-\beta_1 w^2+\beta_2 (v_\xi+w)u\, ,\, \xi W\big)\end{matrix}\right)\ ,
\eqn
where
$$
\Wqh:=\{w\in\Wq\,;\, w(a,x)>-1/2\gamma\,,\, a\in J\,,\, x\in\bar{\Om}\}
$$
is an open subset of $\Wq$ owing to \eqref{emb}.
Clearly, $F$ is smooth and $F(\eta,0,0)=0$ for $\eta\in\R$. As for the Frech\'et derivatives at $(\eta,u,w)$ we compute
\bqn\label{Z1}
F_{(u,w)}(\eta,u,w)[\phi,\psi]=\left(\begin{matrix} \phi-T_1\big( \Delta_D(\gamma \psi u) +\Delta_D(\gamma w\phi)-2\alpha_1 u\phi-\alpha_2 w\phi-\alpha_2 u \psi \, ,\, \eta \Phi\big)\\ \psi-T_2\big( -2\beta_1 w\psi+\beta_2 \psi u+\beta_2 (v_\xi+w)\phi \, ,\, \xi \Psi\big)
\end{matrix}\right)
\eqn
and
\bqn\label{Z2}
F_{\eta,(u,w)}(\eta,u,w)[\phi,\psi]=\left(\begin{matrix} -T_1(0\, , \, \Phi)\\ 0\end{matrix}\right)
\eqn
for $(\phi,\psi)\in \Wq\times\Wq$. The choice of $\Wq$ as solution space is 
to have a suitable functional setting to work with in the framework of maximal regularity. However, as it is needed later on, we note that solutions to \eqref{7}-\eqref{6a}, i.e. to \eqref{1ab}-\eqref{4ab}, are smooth. The proof is a bootstrapping argument which we provide for the reader's ease.

\begin{lem}\label{A44}
If $(\eta_j,u_j,v_j)$ is a bounded sequence in $\R\times\Wq\times\hat{\mathbb{W}}_q$ of solutions to \eqref{7}-\eqref{6a}, then $(u_j)$ and $(v_j)$ are bounded in $C^\ve(J,C^{2+\ve}(\bar{\Om}))\cap C^{1+\ve}(J,C^{\ve}(\bar{\Om}))$
for some $\ve>0$.
\end{lem}

\begin{proof}
To stick with the notation of \cite{LQPP}, let $(E,A):=(L_q,-\Delta_D)$ and let $[(E_\alpha,A_\alpha);\alpha\ge 0]$ be the corresponding interpolation scale induced by the real interpolation functors $(\cdot, \cdot)_{\alpha , q}$. Putting 
$$
F_0:= E_{1-1/q}\doteq\Wqb^{2-2/q}\ ,\qquad F_1:=E_{2-1/q}\ , 
$$
it follows from \cite[V.Thm.2.1.3]{LQPP} that the $F_0$-realization of $\Delta_D$, again denoted by $\Delta_D$, has domain $F_1$ and is the generator of an analytic semigroup $\{e^{a\Delta_D}\,;\,a\ge 0\}$ on $F_0$. Thus,
\bqn\label{51B}
\|e^{a\Delta_D}\|_{\mathcal{L}(F_\mu, F_\nu)}\le c_0 a^{\mu-\nu}\ ,\quad a\in J\setminus\{0\}\ ,\quad \mu,\nu\in (0,1)\ ,
\eqn
where $F_\mu:=(F_0,F_1)_{\mu, q}$ for $\mu\in (0,1)$. Note that the almost reiteration property \cite[V.Thm.1.5.3]{LQPP} ensures
\bqn\label{52}
F_{\theta^+}\hookrightarrow E_{1+\theta-1/q}\hookrightarrow F_{\theta^-}\ ,\quad 0<\theta^-<\theta<\theta^+<1\ .
\eqn
Let now  $(\eta_j,u_j,v_j)$ be a sequence of solutions to \eqref{7}-\eqref{6a} in $\R\times\Wq\times\hat{\mathbb{W}}_q$ with $\vert\eta_j\vert+\|(u_j,v_j)\|_{\X_1}\le B$, $j\in\N$, for some $B>0$.
Writing 
\bqn\label{vj}
\partial_a v_j-\Delta_D v_j=-\beta_1 v_j^2+\beta_2 v_ju_j=:f_j\ ,\quad v_j(0)=\xi V_j=: v_j^0\ ,
\eqn
it follows from the continuity of pointwise multiplication $\Wq\times\Wq\rightarrow\Wq$ (owing to $q>n+2$ and Sobolev's embedding) and \eqref{emb} that
\bqn\label{51}
\|f_j\|_{C(J,F_0)}\le c(B)\ ,\quad j\in\N\ ,
\eqn
while \eqref{7}, \eqref{10}, and the embedding $E_1\hookrightarrow F_{1/q-\ve}$ with $\ve>0$ sufficiently small entail
\bqn\label{511}
\|v_j^0\|_{F_{1/q-\ve}}\le c(B)\ ,\quad j\in\N\ ,
\eqn
for some constant $c(B)>0$. Thus, from \eqref{51B}, \eqref{vj}, \eqref{51}, and \eqref{511}, for $j\in\N$,
\bqnn
\begin{split}
\|v_j\|_{L_1(J,F_{1-\ve})} &\le \int_0^{a_m}\|e^{a\Delta_D}\|_{\mathcal{L}(F_{1/q-\ve}, F_{1-\ve})}\|v_j^0\|_{F_{1/q-\ve}}\,\rd a\\
&\quad + \int_0^{a_m}\int_0^a \|e^{(a-\sigma)\Delta_D}\|_{\mathcal{L}(F_{0}, F_{1-\ve})} \,\|f_j(\sigma)\|_{F_{0}}\,\rd \sigma\,\rd a\\
&\le c(B)\ .
\end{split}
\eqnn
Therefore, from \eqref{10} we conclude that $(v_j^0)$ is bounded in $F_{1-\ve}$. Since $(f_j)$ is bounded in $C(J,F_0)$, we deduce from \eqref{vj} and \cite[II.Thm.5.3.1]{LQPP} that $(v_j)$ is bounded in $C^\ve (J,F_{1-2\ve})$ for some $\ve>0$ sufficiently small. Now, taking \cite[Thm.5.3.4,Thm.5.4.1]{Triebel} into account which guarantee
$$
E_{2-1/q}\doteq\big(D(\Delta_D),D(\Delta_D^2)\big)_{1-1/q,q}\hookrightarrow \Wqb^{2+2(1-1/q)}\ ,
$$
with $D(\Delta_D^k)$ denoting the domain of the $k$-th power of $\Delta_D$ equipped with its graph norm, we obtain
$$
F_{1-2\ve}=(E_{1-1/q}, E_{2-1/q})_{1-2\ve,q}\hookrightarrow \big( \Wqb^{2(1-1/q)},  \Wqb^{2+2(1-1/q)}\big)_{1-2\ve,q}\doteq \Wqb^{4-2/q-4\ve}\hookrightarrow C^{2+\ve}(\bar{\Om})
$$
for $\ve>0$ sufficiently small by Sobolev's embedding theorem since $q>n+2$. Consequently, 
\bqn\label{bj}
(v_j)\ \text{ is bounded in}\ C^\ve (J,C^{2+\ve}(\bar{\Om}))\cap \hat{\mathbb{W}}_q\ .
\eqn 
But then, since
\bqn\label{u}
\partial_a u_j-\Delta_D\big((1+\gamma v_j)u_j\big) =-\alpha_1u_j^2-\alpha_2 u_j v_j\ ,\quad u_j(0)=\eta_j U_j\ ,
\eqn
we similarly conclude that $(u_j)$ is bounded in $C^\ve (J,C^{2+\ve}(\bar{\Om}))$, where the analogue of \eqref{51B} holds due to \eqref{bj} and \cite[II.$\S$5.1]{LQPP}. Finally, these observations warrant that the sequence $(\Delta_D v_j)$ is bounded in $C^\ve (J,C^{\ve}(\bar{\Om}))$ while $(-\beta_1 v_j^2-\beta_2 v_ju_j)$ is bounded in $C^\ve (J,C^{2+\ve}(\bar{\Om}))$. From \eqref{vj} we derive that $(\partial_a v_j)$ is bounded in $C^\ve (J,C^{\ve}(\bar{\Om}))$ and similarly we derive this for $(\partial_a u_j)$.
\end{proof}

Noticing that $C^\ve(J,C^{2+\ve}(\bar{\Om}))$ embeds compactly in $C^{\bar{\ve}}(J,C^{2+\bar{\ve}}(\bar{\Om}))$ and $C^{1+\ve}(J,C^{\ve}(\bar{\Om}))$ in $C^{1+\bar{\ve}}(J,C^{\bar{\ve}}(\bar{\Om}))$ for $\bar{\ve}\in(0,\ve)$, we deduce:

\begin{cor}\label{C1}
Any bounded and closed subset of $\{(\eta,u,w)\in\R\times\Wq\times\Wqh\,;\, F(\eta,u,w)=0\}$ is compact.
\end{cor}

Let now $(\eta,u,w)\in\R\times\Wq\times\Wqh$ be fixed. We shall show that 
$$
L:=F_{(u,w)}(\eta,u,w)\in\mathcal{L}(\Wq\times\Wq)
$$ 
is a Fredholm operator. To this end, we introduce, for $a\in J$, the operators $A_{ij}(a)\in\mathcal{L}(\Wqb^2,L_q)$ by
\begin{align*}
A_{11}(a)\phi:&=-\Delta_D\big((1+\gamma v_\xi(a)+\gamma w(a))\phi\big)+\alpha_2(v_\xi(a)+w(a))\phi+2\alpha_1 u(a) \phi\   ,\\
A_{12}(a)\psi:&=-\Delta_D\big(\gamma u(a)\psi\big)+\alpha_2 u(a)\psi\   ,\\
A_{21}(a)\phi:&=-\beta_2(v_\xi(a)+w(a))\phi\   ,\\
A_{22}(a)\psi:&=-\Delta_D\psi+2\beta_1(v_\xi(a)+w(a))\psi-\beta_2\psi u(a)\   ,
\end{align*}
for $a\in J$, $\phi, \psi\in\Wqb^2$
and set
$$
\mathbb{A}(a):=\left[\begin{matrix} A_{11}(a)& A_{12}(a)\\ A_{21}(a) & A_{22}(a)
\end{matrix}\right]\ ,\quad a\in J\ .
$$
Moreover, we define
$$
\mathbb{D}(a)h:=\left(\begin{matrix}
-\Delta_D\big((1+\gamma v_\xi(a))h_1\big)+\alpha_2 v_\xi(a) h_1\\ -\Delta_D h_2+2\beta_1v_\xi(a) h_2
\end{matrix}
\right) 
\ ,\qquad h=(h_1,h_2)\in\Wqb^2\times\Wqb^2\ ,\quad a\in J\ ,
$$
so that $\mathbb{D}\in\mathcal{L}(\Wqb^2\times\Wqb^2,L_q\times L_q)$ for $a\in J$,
and we also define $\ell[\eta]\in\mathcal{L}(\Wq\times\Wq, \Wqb^2\times \Wqb^2)$ by
$$
\ell[\eta]z:=\left(\begin{matrix}
\eta\Phi \\ \xi\Psi
\end{matrix}
\right) 
\ ,\quad z=(\phi,\psi)\in\Wq\times\Wq\ .
$$
It then readily follows from \eqref{Z1} that, given $z=(\phi,\psi)$ and $h=(h_1,h_2)$ in $\Wq\times\Wq$, the equation $Lz=h$ is equivalent to
\bqn\label{66}
\partial_a z+\A(a) z=\partial_a h+\mathbb{D}(a)h\ ,\quad a\in J\ ,\qquad z(0)=\ell [\eta] z+h(0)\ .
\eqn
In the sequel, we use the notation
$$
\X_0:=\Lq\times\Lq\ ,\quad \X_1:=\Wq\times\Wq\ ,\quad X_\theta:=\Wqb^{2\theta}\times\Wqb^{2\theta}\ ,\quad \theta\in [0,1]\ .
$$
Let us first observe that

\begin{rem}\label{norm}
The space $\X_1$ can be equipped with an equivalent norm, which is continuously differentiable at all points except zero.
\end{rem}

\begin{proof}
According to \cite{Restrepo}, since $\X_1=\Wq\times\Wq$ is separable, the statement is equivalent to say that the dual space \mbox{$\X_1'=\Wq'\times\Wq'$} of $\X_1$ is separable. But, since $\Wq$ is dense in $\Lq$, the separable space $\Lq'=\mathbb{L}_{q'}$ is dense in $\Wq'$, where $1/q+1/q'=1$. So $\X_1'$ is separable.
\end{proof}

Investigation of \eqref{66} requires the following information on the involved operators:

\begin{lem}\label{L34}
The above defined operators
$(\partial_a+\A,\gamma_0)$ and $(\partial_a+\mathbb{D},\gamma_0)$ both belong to $\mathcal{L}is(\X_1,\X_0\times X_{1-1/q})$, and $\ell[\eta]$ belongs to $\mathcal{K}(\X_1,X_{1-1/q})$.
\end{lem}

\begin{proof}
Writing
\bqnn
\begin{split}
A_{12}(a)\psi &=-\Delta_D\big(\gamma u(a)\psi\big)+\alpha_2 u(a)\psi\\
& = -\divv_x\big(\gamma u(a)\nabla_x\psi\big)+\big\{\alpha_2 u(a)\psi -\divv_x\big(\psi\gamma \nabla_x u(a)\big)\big\}
\end{split}
\eqnn
and using \eqref{emb}, it is readily seen that $\A$ can be written in the form
$$
\mathbb{A}:=\mathbb{A}_1+\mathbb{A}_2:=\left[\begin{matrix} A_{11}& \tilde{A}_{12}\\ 0 & A_{22}
\end{matrix}\right]+\left[\begin{matrix} 0& \hat{A}_{12}\\ A_{21} & 0
\end{matrix}\right]\ 
$$
with 
\bqn\label{Al}
\mathbb{A}_1\in C(J,\mathcal{L}(X_1,X_0))\ ,\quad \mathbb{A}_2\in L_q(J,\mathcal{L}(X_{1-1/q}, X_0))\ .
\eqn 
Recalling 
$$
1+\gamma(v_\xi(a,x)+w(a,x))\ge 1/2\ ,\quad (a,x)\in J\times\bar{\Om}\ ,
$$
due to the positivity of $v_\xi$ and $w\in\Wqh$, it follows as in \eqref{T11} that $A_{11}(a_0)$ and $A_{22}(a_0)$ have maximal $L_q$-regularity for each fixed $a_0\in J$. 
Consequently, the problem
\begin{align*}
 \partial_a z_1 +A_{11}(a_0) z_1 +\tilde{A}_{12}(a_0)z_2&=f_1(a)\ ,\quad a\in J\ ,\qquad z_1(0)=z_1^0\ ,\\
 \partial_a z_2 +A_{22}(a_0) z_2&=f_2(a)\ ,\quad a\in J\ ,\qquad z_2(0)=z_2^0\ ,
\end{align*}
admits for each $f=(f_1,f_2)\in \X_0$ and $z^0=(z_1^0,z_2^0)\in X_{1-1/q}$ a unique solution $z=(z_1,z_2)\in\X_1$ given by
\begin{align*}
& z_1=\big(\partial_a +A_{11}(a_0),\gamma_0\big)^{-1}(f_1-\tilde{A}_{12}(a_0)z_2,z_1^0)\ ,\\
& z_2=\big(\partial_a +A_{22}(a_0),\gamma_0\big)^{-1}(f_2,z_2^0)\ ,
\end{align*}
and there is some constant $c$ independent of $f$ and $z^0$ such that
$$
\|z\|_{\X_1}\le c\big(\|f\|_{\X_0}+\|z^0\|_{X_{1-1/q}})\ .
$$
Therefore, $(\partial_a+\A_1(a_0),\gamma_0)\in\mathcal{L}is(\X_1,\X_0\times X_{1-1/q})$ for each $a_0\in J$, whence $(\partial_a+\A,\gamma_0)\in\mathcal{L}is(\X_1,\X_0\times X_{1-1/q})$ by \eqref{Al} and Lemma~\ref{Jan}. Analogously we deduce the statement on $(\partial_a+\mathbb{D},\gamma_0)$. Since $\Wqb^2$ embeds compactly in $\Wqb^{2-2/q}$, the assertion on $\ell[\eta]\in\mathcal{L}(\X_1,X_1)$ is immediate.
\end{proof}

Based on Lemma~\ref{L34}, we have
$$
\Sigma:= (\partial_a+\A,\gamma_0)^{-1}\in\mathcal{L}(\X_0\times X_{1-1/q},\X_1)\qquad\text{and}\qquad Q_0 :=\big[w\mapsto\ell [\eta]\big(\Sigma(0,w)\big)\big]\in\mathcal{K}(X_{1-1/q})\ .
$$
We now show that $L$ is indeed a Fredholm operator. The proof is along the lines of \cite[Lem.2.1]{WalkerSIMA}.

\begin{prop}\label{pA1}
Let $(\eta,u,w)\in\R\times\Wq\times\Wqh$ and $L=F_{(u,w)}(\eta,u,w)\in\mathcal{L}(\X_1)$. Then $L$ is a Fredholm operator of index zero. More precisely,
\bqn\label{rg}
\im(L)=\big\{h\in\X_1\,;\, h(0)+\ell[\eta](\Sigma(\partial_a h+\mathbb{D} h, 0))\in \im(1-Q_0)\big\}\ 
\eqn
is closed in $\X_1$ and
$$
\kk(L)=\big\{\Sigma(0,w)\,;\ w\in\kk(1-Q_0)\big\}
$$
with
$$
\dimens(\kk(L))=\codim(\im(L))=\dimens(\kk(1-Q_0))<\infty\ .
$$
\end{prop}

\begin{proof}
Owing to \eqref{66} and Lemma~\ref{L34}, for $z,h\in \X_1$, the 
equation
$Lz=h$
is equivalent to
\begin{align}
z&=\Sigma(\partial_a h+\mathbb{D} h, 0)+\Sigma(0,z(0))\ ,\label{17f}\\
(1-Q_0)z(0)&= \ell[\eta]\big(\Sigma(\partial_a h+\mathbb{D} h, 0)\big)+ h(0)\ .\label{18f}
\end{align}
If $1$ belongs to the resolvent set of $Q_0\in\mathcal{K}(X_{1-1/q})$, then \eqref{17f}, \eqref{18f} entail a trivial kernel $\kk(L)$. Moreover, in this case, for an arbitrary $h\in \X_1$, there is a unique $z(0)\in X_{1-1/q}$ solving \eqref{18f}, thus the corresponding $z\in\X_1$ given by \eqref{17f} is the unique solution to $Lz=h$. This easily gives the assertion in this case.

Otherwise, if $1$ is an eigenvalue of $Q_0\in\mathcal{K}(X_{1-1/q})$, then \eqref{17f}, \eqref{18f} yield the characterization of $\kk(L)$ and $\im(L)$ as stated. In particular, 
since $\Sigma$ is an isomorphism, we deduce $\mathrm{dim}(\kk(L))=\mathrm{dim}(\kk(1-Q_0))$ which is a finite number because 1 is an eigenvalue of the compact operator $Q_0$. Moreover, $\im(L)$ is closed in $\X_1$ since $M:=\im(1-Q_0)$ is closed by the compactness of $Q_0$ and due to Lemma~\ref{L34} and \eqref{emb}.
To compute $\codim(\im(L))$, note that $$\codim(M)=\dimens(\kk(1-Q_0))<\infty\ ,$$ hence $M$ is complemented in $X_{1-1/q}$ leading to a direct sum decomposition $X_{1-1/q}=M\oplus N$. Denoting by \mbox{$P_M\in\ml(X_{1-1/q})$} a projection onto $M$ along $N$, we set
\bqn\label{18B}
\mathcal{P}h:=\Lambda\big(\partial_a h+\mathbb{D}h\,,\, P_Mh (0)-(1-P_M)\ell[\eta](\Sigma(\partial_a h+\mathbb{D}h , 0))\big)\ ,\quad h\in\X_1\ ,
\eqn
where $\Lambda:=(\partial_a +\mathbb{D},\gamma_0)^{-1}\in\mathcal{L}(\X_0\times X_{1-1/q},\X_1)$, and obtain $\mathcal{P}\in\mathcal{L}(\X_1)$ from Lemma~\ref{L34}. Since
$$
\big(\partial_a +\mathbb{D}\big)(\mathcal{P}h)=\partial_a h+\mathbb{D}h
\ ,\qquad \gamma_0(\mathcal{P}h)=  P_Mh (0)-(1-P_M)\ell[\eta](\Sigma(\partial_a h+\mathbb{D}h , 0))\ ,
$$
the characterization \eqref{rg} actually implies that $\mathcal{P}$ maps $\X_1$ into $\im(L)$. Furthermore, if $h\in \im(L)$, then \eqref{rg} also ensures  $$\mathcal{P}h=\Lambda(\partial_ah +\mathbb{D}h,h(0))=h\ ,$$ so  $\mathcal{P}(\im(L))=\im(L)$. Thus $\mathcal{P}^2=\mathcal{P}$ with $\im(\mathcal{P})=\im(L)$ is a projection and $\X_1=\im(L)\oplus\kk (\mathcal{P})$. Since $\Lambda$ is an isomorphism, we obtain 
$$
\kk (\mathcal{P})=\{h\in\X_1\,;\, \partial_ah +\mathbb{D}h=0, h(0)\in N\}\ ,
$$
from which we deduce the equality of the dimension of $N$ and $\kk (\mathcal{P})$ and thus the statement.
\end{proof}

\begin{cor}\label{AA1}
For $k\in (0,1)$ and $(\eta,u,w)\in\R\times\Wq\times\Wqh$,
$$
(1-k)F_{(u,w)}(\eta,0,0)+k F_{(u,w)}(\eta,u,w)\in\mathcal{L}(\X_1)$$ is a Fredholm operator of index zero.
\end{cor}

\begin{proof}
Since, by \eqref{Z1},
\bqn\label{Z11}
F_{(u,w)}(\eta,0,0) [\phi,\psi]=\left(\begin{matrix} \phi-T_1\big(0 \, ,\, \eta \Phi\big)\\ \psi-T_2\big( \beta_2 v_\xi\phi \, ,\, \xi \Psi\big)
\end{matrix}\right)\ ,
\eqn
the operator $$(1-k)F_{(u,w)}(\eta,0,0)+k F_{(u,w)}(\eta,u,w)$$ has the same structure as $F_{(u,w)}(\eta,u,w)$.
\end{proof}

It follows from Lemma~\ref{A44} that the operator $A_\xi$, given by
\bqn\label{Axi}
A_\xi(a) \phi:=-\Delta_D\big( (1+\gamma v_\xi(a))\phi\big) +\alpha_2 v_\xi (a) \phi\ ,\quad a\in J\ ,\quad \phi\in\Wqb^2\ ,
\eqn
belongs to $C^\ve(J,\mathcal{L}(\Wqb^2,L_q))$, while the positivity of $v_\xi$ ensures that $-A_\xi(a)$ is for each $a\in J$ the generator of an analytic semigroup on $L_q$. Consequently, it generates a parabolic evolution operator $\Pi_{A_\xi}(a,\sigma)$, $0\le\sigma\le a\le a_m$, in view of \cite[II.Cor.4.4.2]{LQPP}. Note that $\Pi_{A_\xi}(a,0)$ for $a>0$ is strongly positive on $\Wqb^{2-2/q}$, see e.g. \cite[Cor.13.6]{DanersKochMedina}. We then set 
\bqn\label{Z111}
G_{\xi}:=\int_0^{a_m}b(a)\,\Pi_{A_\xi}(a,0)\,\rd a
\eqn
and obtain from \eqref{10} and the compact embedding of $\Wqb^2$ in $\Wqb^{2-2/q}$ that  $G_\xi\in\mk(\Wqb^{2-2/q})$ is strongly positive. Thus, by the Krein-Rutman theorem, $r(G_\xi)>0$ is a simple eigenvalue of $G_\xi$ with an eigenvector in the interior of the positive cone $\Wqb^{2-2/q,+}$. Let then
\bqn\label{eta0}
\eta_0:=\eta_0(\xi):=\frac{1}{r(G_{\xi})}>0\ ,\quad \kk(1-\eta_0G_\xi)=\mathrm{span}\{\Phi_0\}\ ,\quad \Phi_0\in \mathrm{int}(\Wqb^{2-2/q,+})\ .
\eqn
 We define
\bqn\label{phistern}
\phi_*(a):=\Pi_{A_\xi}(a,0)\Phi_0\ ,\quad a\in J\ ,\qquad \Phi_*:=\int_0^{a_m}b(a)\phi_*(a)\,\rd a\ ,
\eqn
and, using the notation of Subsection~\ref{sec 2.1},
\bqn\label{psistern}
\psi_*(a):=\Pi_{[2\beta_1 v_\xi]}(a,0)\Psi_0+\int_0^a \Pi_{[2\beta_1 v_\xi]}(a,\sigma)\,\big(\beta_2 v_\xi (\sigma)\phi_*(\sigma)\big)\,\rd \sigma\ ,\quad a\in J\ ,
\eqn
where 
$$
\Psi_0:=\xi \big(1-\xi H_{[2\beta_1 v_\xi]}\big)^{-1}\left(\int_0^{a_m} b(a)\int_0^a \Pi_{[2\beta_1 v_\xi]}(a,\sigma)\,\big(\beta_2 v_\xi (\sigma)\phi_*(\sigma)\big)\,\rd \sigma\,\rd a\right)\ .
$$
Note that $\Psi_0$ is well-defined since $1-\xi H_{[2\beta_1 v_\xi]}$ is invertible owing to Lemma~\ref{A2}, \eqref{sp}, and $v_\xi\in\Wqd$ which ensure $r(\xi  H_{[2\beta_1 v_\xi]})<1$. Also note, from \eqref{maxreg} and \eqref{T11}, that $\phi_*$ and $\psi_*$ both belong to $\Wqd$. 

\begin{lem}\label{A33}
The kernel of $F_{(u,w)}(\eta_0,0,0)$ is spanned by $(\phi_*,\psi_*)$, and $F_{\eta, (u,w)}(\eta_0,0,0)[\phi_*,\psi_*]$ does not belong to the range of $F_{(u,w)}(\eta_0,0,0)$. Moreover,  $\Phi_0=\eta_0\Phi_*$.
\end{lem}

\begin{proof}
Observe that $(\phi,\psi)\in\X_1$ belonging to the kernel of $F_{(u,w)}(\eta_0,0,0)$ is equivalent to
\begin{align*}
\partial_a \phi-\Delta_D\big((1+\gamma v_\xi)\phi\big)+\alpha_2 v_\xi\phi &=0\ ,& \!\!\!\!\!\!\!\!\!\!\!\!\!\!\!\!\phi(0)=\eta_0\Phi\ ,\\
\partial_a \psi-\Delta_D \psi+2\beta_1v_\xi\psi&=\beta_2 v_\xi\phi\ ,&  \psi(0)=\xi\Psi\ , 
\end{align*}
according to \eqref{Z11} and the definitions of $T_1$ and $T_2$. Now, the first assertion follows from \eqref{Axi}-\eqref{psistern} by solving for $\phi$ and $\psi$. Next, suppose $F_{\eta, (u,w)}(\eta_0,0,0)[\phi_*,\psi_*]$ belongs to the range of $F_{(u,w)}(\eta_0,0,0)$. Then, in view of \eqref{Z2}, \eqref{Z11}, and the definition of $T_1$, there is $\phi\in\Wq$ with
$$
\partial_a\phi+A_\xi\phi=0\ ,\quad \phi(0)=\eta_0\Phi-\Phi_*\ ,
$$
so
$
\phi(a)=\Pi_{A_\xi}(a,0)\phi(0)$, $a\in J$ and whence
$
(1-\eta_0G_\xi)\phi(0)=-\Phi_*$.
Since $\Phi_0=\eta_0\Phi_*$ by definition of $\phi_*$ and \eqref{eta0}, we conclude
$$
\Phi_0\in\kk(1-\eta_0G_\xi)\cap\im(1-\eta_0G_\xi)
$$
what is impossible since $\eta_0G_\xi$ is compact with simple eigenvalue $1$.
\end{proof}

\subsection{Proof of Theorem~\ref{T1}}\label{sec 2.3}

Having established the necessary auxiliary results in the previous subsection, we are now in a position to prove Theorem~\ref{T1} by applying \cite[Thm.4.3,Thm.4.4]{ShiWang}. Recall that, writing
$(\eta,u,v)=(\eta,u,v_\xi+w)$, the solutions $(\eta,u,v)$ to \eqref{7}-\eqref{6a} are obtained as the zeros $(\eta,u,w)$ of the smooth function $F$ defined in \eqref{F}. Also recall that $\eta_0=\eta_0(\xi)$ is given in \eqref{eta0}.

As in the second part of the proof of Lemma~\ref{A33}, 
$$
\kk(F_{(u,w)}(\eta_0,0,0))\cap \im(F_{(u,w)}(\eta_0,0,0))=\{0\}\ ,
$$
whence $$\X_1=\mathrm{span}\{(\phi_*,\psi_*)\}\oplus \im(F_{(u,w)}(\eta_0,0,0))$$  by \cite[Lem.2.7.9]{BuffoniToland} and Lemma~\ref{A33}.
In view of  Proposition~\ref{pA1} and Lemma~\ref{A33} we may apply \cite[Thm.4.3]{ShiWang}. Therefore, there are $\ve>0$ and continuous functions 
$$
\eta:(-\ve,\ve)\rightarrow\R\ ,\qquad (\theta_1,\theta_2):(-\ve,\ve)\rightarrow \im(F_{(u,w)}(\eta_0,0,0))
$$ 
with \mbox{$\eta(0)=\eta_0$} and $(\theta_1,\theta_2)(0)=(0,0)$ such that the solutions to \eqref{7}-\eqref{6a} near $(\eta_0,0,v_\xi)$ are exactly the semi-trivial ones $(\tilde{\eta},0,v_\xi)$, $\tilde{\eta}\ge 0$, and the ones lying on the curve
$\Gamma:=\Gamma_+\cup\Gamma_-\cup\{(\eta_0,0,v_\xi)\}$,
where
$$
\Gamma_\pm:=\big\{\big(\eta(s),s\phi_*+s\theta_1(s), v_\xi+s\psi_*+s\theta_2(s)\big)\,;\, 0<\pm s<\ve\big\}\ .
$$
Moreover, $\Gamma$ is contained in $\mathfrak{C}_*$, which is a connected component of the closure of $$
\mathcal{S}:=\{(\eta,u,v_\xi+w)\,;\, F(\eta,u,w)=0\, ,\, (u,w)\not= (0,0)\}\ .
$$
Being merely interested in positive solutions, we first note:
\begin{lem}\label{pos}
The curve $\Gamma_+$ lies in $\R^+\times\Wqd\times\Wqd$.
\end{lem}

\begin{proof}
Let $u^s:=s\phi_*+s\theta_1(s)$ and $v^s:=v_\xi+s\psi_*+s\theta_2(s)$. Then $u^s(0)=s\Phi_0+o(s)$ and $v^s(0)=\xi V_\xi+s\Psi_0+o(s)$ in $\Wqb^{2-2/q}$ as $s\rightarrow0^+$ by \eqref{emb}. Thus, it follows from $\Phi_0, V_\xi\in \mathrm{int}(\Wqb^{2-2/q,+})$ that $u^s(0), v^s(0)\in \mathrm{int}(\Wqb^{2-2/q,+})$ for $s\in (0,\ve)$ with $\ve>0$ small enough, whence $u^s, v^s\in\Wqd$  for $s\in (0,\ve)$ due to the parabolic maximum principle \cite[Thm.13.5]{DanersKochMedina} and \eqref{7}-\eqref{6a}. 
\end{proof}

\noindent Now, invoking Corollary~\ref{C1}, Remark~\ref{norm}, and Corollary~\ref{AA1} we obtain from \cite[Thm.4.4]{ShiWang} (see also \cite[Rem.4.2.1]{ShiWang}) further information about the global character of the continuum. More precisely, if $\mathfrak{C}_+$ denotes the connected component of $\mathfrak{C}_*\setminus\Gamma_-$ containing $\Gamma_+$, then $\mathfrak{C}_+$ satisfies the alternatives:
\begin{itemize}
\item[(i)] $\mathfrak{C}_+$ intersects with the boundary of $\R\times\Wq\times\Wqh$, or
\item[(ii)] $\mathfrak{C}_+$ is unbounded in $\R\times\Wq\times\Wqh$, or
\item[(iii)] $\mathfrak{C}_+$ contains a point $(\eta,0,v_\xi)$ with $\eta\not= \eta_0$, or
\item[(iv)] $\mathfrak{C}_+$ contains a point $(\eta,u,v_\xi+w)$ with $(u,w)\not=(0,0)$ and $(u,w)\in \im(F_{(u,w)}(\eta_0,0,0))$.
\end{itemize}

Due to Lemma~\ref{pos}, the continuum $\mathfrak{C}:=\mathfrak{C}_+\cap (\R^+\times\Wq^+\times\Wq^+)$ of solutions to \eqref{7}-\eqref{6a} contains the curve~$\Gamma_+$. Furthermore, we have:

\begin{lem}\label{C}
$\mathfrak{C}\setminus\{(\eta_0,0,v_\xi)\}\subset\R^+\times\Wqd\times\Wqd$ is unbounded.
\end{lem}

\begin{proof}
We first show that $\mathfrak{C}_+$ does not reach the boundary of $\R^+\times\Wqd\times\Wqd$ at some point $(\eta,u,v)\not= (\eta_0,0,v_\xi)$. Suppose otherwise, i.e. let there be a sequence $(\eta_j,u_j,v_j)$ in $\mathfrak{C}\cap (\R^+\times\Wqd\times\Wqd)$ converging toward some point $(\eta,u,v)\not= (\eta_0,0,v_\xi)$ not belonging to $\R^+\times\Wqd\times\Wqd$. Since $u_j$ and $v_j$ are nonnegative, the limits $u$ and $v$ are as well. So $u\equiv0$ or $v\equiv 0$ because $(\eta,u,v)\notin\R^+\times\Wqd\times\Wqd$. We claim that neither is possible. Suppose first that both $u$ and $v$ identically vanish. As $v_j\in \Wqd$, $\psi_j:=v_j/\|v_j\|_{\Wq}$ is well-defined in $\Wqd$, has norm 1, and 
$$
\partial_a \psi_j-\Delta_D \psi_j=-\beta_1 v_j\psi_j+\beta_2 \psi_ju_j\ ,\quad \psi_j(0)=\xi \Psi_j\ .
$$
The proof of Lemma~\ref{A44} shows that $(\psi_j)$ is bounded in $C^\ve(J,C^{2+\ve}(\bar{\Om}))\cap C^{1+\ve}(J,C^{\ve}(\bar{\Om}))$
for some $\ve>0$ and so we may assume without loss of generality that $(\psi_j)$ converges in $\Wqd$ to some $\psi$ satisfying
$$
\partial_a \psi-\Delta_D\psi =0\ ,\quad \psi(0)=\xi \Psi\ .
$$
Thus $\psi(a)=e^{a\Delta_D}\psi(0)$, $a\in J$, and $\psi(0)=\xi H_{[0]} \psi(0)$ implying  $\xi r(H_{[0]})=1$ by the Krein-Rutman theorem. However, this contradicts \eqref{1mio} and $\xi>1$. Next, assume $u$ vanishes identically but $v\not\equiv 0$. Then $(\eta,u,v)=(\eta,0,v)$ and the uniqueness statement of Theorem~\ref{A1} implies $v=v_\xi$. Thus, $( \eta,0,v_\xi)\in\mathfrak{B}_2$ is a bifurcation point to positive coexistence states. By Lemma~\ref{A44}, we may assume $(v_j)$ converges to $v_\xi$ in $C^\ve(J,C^{2+\ve}(\bar{\Om}))\cap C^{1+\ve}(J,C^{\ve}(\bar{\Om}))$
for some $\ve>0$. Moreover, as above we may assume that $(\phi_j)$, defined by $\phi_j:=u_j/\|u_j\|_{\Wq}$, converges in $\Wqd$ to some $\phi$ satisfying
\bqn\label{Z0}
\partial_a \phi-\Delta_D\big((1+\gamma v_\xi)\phi\big) =-\alpha_2 \phi v_\xi\ ,\quad \phi(0)=\eta \Phi\ .
\eqn
Therefore, $\phi(a)=\Pi_{A_\xi}(a,0)\phi(0)$, $a\in J$, and $\phi(0)=\eta G_\xi \phi(0)$. Thus $\eta=\eta_0$ by the Krein-Rutman theorem and \eqref{eta0}. This yields the contradiction $(\eta,u,v)=(\eta_0,0,v_\xi)$. 
Finally, suppose $v\equiv 0$ but $u\not\equiv 0$. Then we have
$(\eta,u,v)=(\eta,u,0)$ what gives $u=u_\eta$ with $\eta>1$ by Theorem~\ref{A1} since $u\in\Wqd$, and so
\mbox{$(\eta,u,v)=(\eta,u_\eta,0)\in\mathfrak{B}_1$} is a bifurcation point to positive coexistence states. As above we may assume that $(\psi_j)$, given by $\psi_j:=v_j/\|v_j\|_{\Wq}$, converges to some $\psi\in\Wqd$ satisfying
$$
\partial_a \psi-\Delta_D \psi=\beta_2 \psi u_\eta\ ,\quad \psi(0)=\xi \Psi\ .
$$
This readily implies $1=\xi r(H_{[-\beta_2 u_\eta]})$ what is impossible since $\xi>1$ and $1=r(H_{[0]})<r(H_{[-\beta_2 u_\eta]})$ according to \eqref{1mio} and Lemma~\ref{A2}. 

Consequently, $\mathfrak{C}_+$ intersects with the boundary of $\R^+\times\Wqd\times\Wqd$ only at $(\eta_0,0,v_\xi)$, whence $\mathfrak{C}=\mathfrak{C}_+$. So neither alternative (i) nor (iii) above is possible. Suppose (iv) occurs, and let $(\phi,\psi)\in\X_1$ and $(\eta,u,v_\xi+w)\in \mathfrak{C}_+$ be with $$(u,w)=F_{(u,w)}(\eta_0,0,0)[\phi,\psi]\ .$$ Then $\phi-u=T_1(0,\eta_0\Phi)$ by \eqref{Z11}. Recall, from the definition of $\phi_*$ and Lemma~\ref{A33}, that $\phi_*=T_1(0,\eta_0\Phi_*)$ with $\Phi_*\in \mathrm{int}(\Wqb^{2-2/q,+})$. The latter implies $\kappa\eta_0\Phi_*+\phi(0)-u(0)\in \mathrm{int}(\Wqb^{2-2/q,+})$ for some $\kappa>0$. Defining \mbox{$p:=\kappa\phi_*+\phi-u\in\Wq$}, we obtain $p=T_1(0,\eta_0(\kappa\Phi_*+\Phi))$, that is,
$
\partial_a p +A_\xi p=0$ with $p(0)=\eta_0 P+\eta_0 U
$.
Hence $(1-\eta_0G_\xi)p(0)=\eta_0 U$. Since $u\in \Wqd$ by assumption and thus $U\in (\Wqb^{2-2/q,+})\setminus\{0\}$, this last equation does not admit a positive solution $p(0)$ according to \eqref{eta0} and \cite[Thm.3.2]{AmannSIAMReview} in contradiction to $p(0)\in\mathrm{int}(\Wqb^{2-2/q,+})$ by the choice of $\kappa$. So (iv) is impossible as well, and we conclude that  $\mathfrak{C}\setminus\{(\eta_0,0,v_\xi)\}\subset\R^+\times\Wqd\times\Wqd$ is unbounded.
\end{proof}

To finish off the proof of Theorem~\ref{T1} we merely have to remark that $(\eta_0,0,v_\xi)$ is the only bifurcation point.

\begin{lem}
There is no other bifurcation point on $\mathfrak{B}_2$ or on $\mathfrak{B}_1$ to positive coexistence states.
\end{lem}

\begin{proof}
Exactly the same arguments as in the first step of the proof of Lemma~\ref{C} show that there is neither a bifurcation point $(\eta,u_\eta,0)\in\mathfrak{B}_1$ nor  $(\eta,0,v_\xi)\in\mathfrak{B}_2$ to positive coexistence states.
\end{proof}

\section{Proof of Theorem~\ref{T22}}\label{sec 3}

\noindent As the proof of Theorem~\ref{T22} is similar to the one of Theorem~\ref{T1}, we merely sketch it and point out the necessary modifications. 

We shall derive a bifurcation from the branch $\mathfrak{B}_1$ by linearizing around a point $(\eta,u_\eta,0)$ with a suitable $\eta=\eta_1$ to be determined. First note that the smooth branch $\mathfrak{U}:=\{( \eta,u_\eta); \eta>1\}$ in $(1,\infty)\times\Wqd$ of solutions to \eqref{Aee1} provided by Theorem~\ref{A1} extends to a smooth branch $\mathfrak{U}_*:=\{( \eta,u_\eta); \eta>\eta_*\}$ in $(\eta_*,\infty)\times\Wq$ passing through $(\eta,u)=(1,0)$, where $\eta_*\in (0,1)$ and $-u_\eta\in\Wqd$ for $\eta\in (\eta_*,1)$. Indeed, application of \cite[Thm.2.4]{WalkerSIMA}, \cite[Prop.2.5]{WalkerJDE} (with $\ve<0$ in \cite[Eq.(2.17)]{WalkerJDE}, see the proof of \cite[Prop.3.4]{WalkerJRAM}) shows that the branch $\mathfrak{U}$ of positive solutions extends smoothly with a branch $\{(\eta(\ve),u_{\eta(\ve)});-\ve_0<\ve\le 0\}$, where $-u_{\eta(\ve)}\in\Wqd$ for $\ve\in(-\ve_0,0)$. Thus, fixing $\ve\in(-\ve_0,0)$, it follows that $w:=-u_{\eta(\ve)}\in\Wqd$ satisfies
$$
\partial_a w-\Delta_Dw=\alpha_1 w^2\ ,\quad w(0)=\eta(\ve) W\ ,
$$
whence $w(0)=\eta(\ve) H_{[- \alpha_1w]} w(0)$ and thus $\eta(\ve) r(H_{[- \alpha_1w]}) =1$ by the Krein-Rutman theorem. Due to Lemma~\ref{A2} and \eqref{1mio}, we have $r(H_{[- \alpha_1w]}) >r(H_{[0]})=1$ and so $\eta(\ve)<1$. 
We thus get the desired smooth extension $\mathfrak{U}_*$ of $\mathfrak{U}$ by choosing $\eta_*$ sufficiently close to $1$.
Consequently, the solutions $(\eta,u,v)=(\eta,u_\eta-w,v)$ to problem \eqref{7}-\eqref{6a} can be obtained as the zeros $(\eta,w,v)$ of the smooth map $\mathcal{F}:(\eta_*,\infty)\times\Wq\times\Wqh\rightarrow \Wq\times\Wq$, defined by
\bqn\label{G}
\mathcal{F}(\eta,w,v):=\left(\begin{matrix} w-T\big(-\Delta_D(\gamma v(u_\eta-w))-2\alpha_1u_\eta w+\alpha_1w^2+ \alpha_2 (u_\eta-w)v\, , \,\eta W\big)\\ v-T\big(-\beta_1 v^2+\beta_2 v(u_\eta-w)\, ,\, \xi V\big)\end{matrix}\right)\ ,
\eqn
where the set $\Wqh$ is as in Section~\ref{sec 2} and
$$
T:=(\partial_a-\Delta_D,\gamma_0)^{-1}\in\mathcal{L}(\Lq\times \Wqb^{2-2/q},\Wq)\ .
$$
Clearly, $\mathcal{F}(\eta,0,0)=0$ for $\eta\in (\eta_*,\infty)$ and the Frech\'et derivatives at $(\eta,w,v)$ are given by
\bqn\label{Z1a}
\mathcal{F}_{(w,v)}(\eta,w,v)[\phi,\psi]=\left(\begin{matrix} \phi-T\big( -\Delta_D(\gamma \psi (u_\eta-w)) +\Delta_D(\gamma v\phi)-2\alpha_1 (u_\eta-w)\phi\\ \qquad\qquad\quad
+\alpha_2 \psi(u_\eta-w)-\alpha_2 v\phi \, ,\, \eta \Phi\big)\\ \psi-T\big( -2\beta_1 v\psi-\beta_2 v  \phi+\beta_2 (u_\eta-w)\psi \, ,\, \xi \Psi\big)
\end{matrix}\right)
\eqn
and
\bqn\label{Z2a}
\mathcal{F}_{\eta,(w,v)}(\eta,u,w)[\phi,\psi]=\left(\begin{matrix} -T\big( -\Delta_D(\gamma \psi u_\eta') -2\alpha_1 u_\eta'\phi+\alpha_2 \psi u_\eta' \, ,\, \Phi\big)\\ -T\big( \beta_2 u_\eta' \psi\, ,\, 0\big)
\end{matrix}\right)
\eqn
for $(\phi,\psi)\in \Wq\times\Wq$ with dashes referring to derivatives with respect to $\eta$. It is then straightforward to modify the proofs of Lemma~\ref{L34} and Proposition~\ref{pA1} in order to derive the analogue of Corollary~\ref{AA1}:

\begin{lem}\label{AAA1}
For $k\in (0,1)$ and $(\eta,w,v)\in (\eta_*,\infty)\times\Wq\times\Wqh$,
$$
(1-k)\mathcal{F}_{(w,v)}(\eta,0,0)+k \mathcal{F}_{(w,v)}(\eta,w,v)\in\mathcal{L}(\X_1)
$$ 
is a Fredholm operator of index zero. 
\end{lem}

To determine the bifurcation point, let us observe that $r(H_{[-\beta_2 u_\eta]})>1$ is a strictly increasing function of $\eta$ on $(1,\infty)$ according to Theorem~\ref{A1} and Lemma~\ref{A2}.
Since $u_\eta$ depends continuously on $\eta$ in the topology of $\Wq$ by Theorem~\ref{A1}, we obtain from \cite[II.Lem.5.1.4]{LQPP} that the evolution operator $\Pi_{[-\beta_2 u_\eta]}(a,0)$ and hence $H_{[-\beta_2 u_\eta]}$ depend continuously on $\eta$ with respect to the corresponding operator topologies. Together with the fact that the spectral radius considered as a function $\mk(\Wqq)\rightarrow \R^+$ is continuous (see \cite[Thm.2.1]{Degla09}), we conclude that
\bqn\label{46}
\big(\eta\mapsto r(H_{[-\beta_2 u_\eta]})\big)\in C\big((1,\infty),(1,\infty)\big)\ \,\text{is strictly increasing}\ 
\eqn
with $\lim_{\eta\rightarrow 1}r(H_{[-\beta_2u_\eta]})=1$. Defining $\delta\in [0,1)$ by
\bqn\label{delta}
\delta:=\frac{1}{\lim\limits_{\eta\rightarrow \infty}r(H_{[-\beta_2 u_\eta]})}\ ,
\eqn
it follows that for any $\xi\in (\delta,1)$ fixed we find a unique $\eta_1:=\eta_1(\xi)>1$ with
\bqn\label{eta1}
\xi=\frac{1}{r(H_{[-\beta_2 u_{\eta_1}]})}\ .
\eqn
We may then choose $\Psi_1\in \mathrm{int}(\Wqqp)$ spanning $\mathrm{ker}\big(1-\xi H_{[-\beta_2 u_{\eta_1}]}\big)$. Define $(\phi_\star,\psi_\star)\in \Wq\times\Wqd$ by
\mbox{$\psi_\star:=\Pi_{[-\beta_2 u_{\eta_1}]}(\cdot,0)\,\Psi_1$}
and
$$
\phi_\star:=\Pi_{[2\alpha_1 u_{\eta_1}]}(\cdot,0)\Phi_1+ N\psi_\star\ ,\qquad \Phi_1:={\eta_1}\big(1-{\eta_1}H_{[2\alpha_1 u_{\eta_1}]}\big)^{-1} \left(\int_0^{a_m} b(a)(N\psi_\star)(a)\,\rd a\right)\  ,
$$
with
$$
(N\psi_\star)(a):=
\int_0^a \Pi_{[2\alpha_1 u_{\eta_1}]}(a,\sigma)\,\big(-\Delta_D(\gamma u_{\eta_1}(\sigma)\psi_\star(\sigma))+\alpha_2u_{\eta_1}(\sigma) \psi_\star(\sigma)\big)\,\rd \sigma\ , \quad a\in J\ ,
$$
where the invertibility of $1-{\eta_1}H_{[2\alpha_1 u_{\eta_1}]}$ is due to \eqref{sp}. 
The analogue of Lemma~\ref{A33} then reads:

\begin{lem}\label{A33a}
The kernel of $\mathcal{F}_{(w,v)}(\eta_1,0,0)$ is spanned by $(\phi_\star,\psi_\star)$ and $\mathcal{F}_{\eta,(w,v)}(\eta_1,0,0)[\phi_\star,\psi_\star]$ does not belong to the range of $\mathcal{F}_{(w,v)}(\eta_1,0,0)$. 
\end{lem}

\begin{proof}
That $\kk(\mathcal{F}_{(w,v)}(\eta_1,0,0))=\mathrm{span}\{(\phi_\star,\psi_\star)\}$ follows as in the proof of Lemma~\ref{A33}. To check the transversality condition, suppose $\mathcal{F}_{\eta,(w,v)}(\eta_1,0,0)[\phi_\star,\psi_\star]$ belongs to the range of $\mathcal{F}_{(w,v)}(\eta_1,0,0)$. Recall \eqref{Z1a}, \eqref{Z2a} and let $v\in\Wq$ be such that $v-T(\beta_2u_{\eta_1}v,\xi V)=-T(\beta_2u_{\eta_1}'\psi_\star,0)$. Choose $\tau>0$ with $\tau\Psi_1-v(0)\in\mathrm{int}(\Wqqp)$. Since $\psi_\star=T(\beta_2 u_{\eta_1}\psi_\star,\xi \Psi_\star)$, it follows that $p:=\tau\psi_\star-v$ satisfies
$$
\partial_a p-\Delta_D p-\beta_2u_{\eta_1} p=\beta_2u_{\eta_1}'\psi_\star\ ,\quad p(0)=\xi P\ ,
$$
from which we deduce 
$$
\big(1-\xi H_{[-\beta_2u_{\eta_1}]}\big) p(0)=\xi\beta_2\int_0^{a_m} b(a)\int_0^a \Pi_{[-\beta_2u_{\eta_1}]}(a,\sigma)\big(u_{\eta_1}'(\sigma)\psi_\star(\sigma) \big)\,\rd \sigma\,\rd a\ .
$$
However, invoking \cite[Thm.3.2]{AmannSIAMReview} and \eqref{eta1}, this is impossible since
$p(0)\in\mathrm{int}(\Wqqp)$ by the choice of $\tau$ and since the right hand side is positive and nonzero due to \eqref{10} and the positivity of $\psi_\star$ and of $u_\eta '$ stated in Theorem~\ref{A1}.
\end{proof}

As Corollary~\ref{C1} holds also for $\mathcal{F}$, we may proceed as in Subsection~\ref{sec 2.3} to derive from \cite[Thm.4.3,Thm.4.4]{ShiWang} that a continuum $\mathfrak{S}^+$ in $(\eta_*,\infty)\times\Wq\times\Wq$ of solutions to \eqref{7}-\eqref{6a} bifurcates from $(\eta_1,u_{\eta_1},0)$ satisfying the alternatives:
\begin{itemize}
\item[(i)] $\mathfrak{S}^+$ intersects with the boundary of $(\eta_*,\infty)\times\Wq\times\Wqh$, or
\item[(ii)] $\mathfrak{S}^+$ is unbounded in $(\eta_*,\infty)\times\Wq\times\Wqh$, or
\item[(iii)] $\mathfrak{S}^+$ contains a point $(\eta,u_\eta,0)$ with $\eta\not= \eta_1$, or
\item[(iv)] $\mathfrak{S}^+$ contains a point $(\eta,u_\eta-w,v)$ with $(w,v)\not=(0,0)$ and $(w,v)\in \im(\mathcal{F}_{(w,v)}(\eta_1,0,0))$.
\end{itemize}
Moreover, near the bifurcation point, $\mathfrak{S}^+$ is a continuous curve
$$
\Gamma^+=\big\{(\eta(s),u_{\eta(s)}-s\phi_\star-s\theta_1(s),s\psi_\star+s\theta_2(s))\,;\, 0<s<\ve\big\}
$$
for a continuous real-valued function $\eta$ and some continuous $\Wq$-valued functions $\theta_j$ with $\eta(0)=0$ and $\theta_j(0)=0$. Since $u_{\eta(s)}(0)$ and $\psi_\star=\Psi_1$ both belong to $\mathrm{int}(\Wqqp)$, it follows as in Lemma~\ref{pos} that $\Gamma^+$ is a subset of $(1,\infty)\times\Wqd\times\Wqd$ for $\ve>0$ sufficiently small. For the continuum $\mathfrak{S}$, given by $$\mathfrak{S}:=\mathfrak{S}^+\cap \big( [\eta_*,\infty)\times\Wq^+\times\Wq^+\big)\ ,$$ we have:

\begin{lem}\label{S}
$
\mathfrak{S}\setminus\{(\eta_1,u_{\eta_1},0)\}$ is a subset of $[\eta_*,\infty)\times\Wqd\times \Wqd
$
consisting of coexistence states $(\eta,u,v)$ to \eqref{7}-\eqref{6a}. The continuum $\mathfrak{S}$ is unbounded or it connects $(\eta_1,u_{\eta_1},0)$ to a solution $(\eta_*,u,v)$ of \eqref{7}-\eqref{6a} with $u,v\in\Wqd$.
\end{lem}

\begin{proof}
First suppose $\mathfrak{S}^+\setminus\{(\eta_1,u_{\eta_1},0)\}$ does not reach the boundary of  $(\eta_*,\infty)\times\Wqd\times\Wqd$, so $\mathfrak{S}=\mathfrak{S}^+$. Then neither (i) nor (iii) above is possible. Suppose (iv) occurs. Then there are $(\eta,u_\eta-w,v)\in \mathfrak{S}$ and \mbox{$(\phi,\psi)\in\Wq\times\Wq$} such that $(w,v)=\mathcal{F}_{(w,v)}(\eta_1,0,0)[\phi,\psi]$. Hence $p:=\kappa\psi_\star+\psi-v\in\Wq$, with $\kappa>0$ chosen such that $p(0)$ belongs to $\in\mathrm{int}(\Wqqp)$, satisfies
$$
\partial_a p -\Delta_D p-\beta_2u_{\eta_1}p=\beta_2u_{\eta_1} v\ ,\quad p(0)=\xi P+\xi V\ ,
$$
so that 
$$
\big(1-\xi H_{[-\beta_2u_{\eta_1}]}\big)p(0)=\xi V+\xi\beta_2\int_0^{a_m} b(a)\int_0^a\Pi_{[-\beta_2u_{\eta_1}]}(a,\sigma)\big(u_{\eta_1}(\sigma) v(\sigma)\big)\,\rd \sigma\,\rd a\ .
$$
Since $v\in \Wqd$ by assumption, this last equation does not admit a positive solution $p(0)$ according to \cite[Thm.3.2]{AmannSIAMReview} in view of \eqref{eta1}. However, this contradicts $p(0)\in\mathrm{int}(\Wqb^{2-2/q,+})$. So (iv) is impossible as well, and we conclude that if $\mathfrak{S}^+\setminus\{(\eta_1,0,v_\xi)\}$ does not reach the boundary of $(\eta_*,\infty)\times\Wqd\times\Wqd$, then $\mathfrak{S}=\mathfrak{S}^+$ is unbounded. Otherwise, suppose $\mathfrak{S}\setminus\{(\eta_1,0,v_\xi)\}$ reaches the boundary of \mbox{$(\eta_*,\infty)\times\Wqd\times\Wqd$} at a point $(\eta,u,v)\not= (\eta_1,u_{\eta_1},0)$ and choose a sequence $(\eta_j,u_j,v_j)$ in $\mathfrak{S}\cap ((\eta_*,\infty)\times\Wqd\times\Wqd)$ converging toward $(\eta,u,v)$. Since $u_j$ and $v_j$ are nonnegative and $\eta_j>\eta_*$, the limits $u$ and $v$ are nonnegative as well and $\eta\ge \eta_*$. So $u\equiv0$ or $v\equiv 0$ or $\eta=\eta_*$. We claim that necessarily $\eta=\eta_*$ and $u,v\in\Wqd$. We proceed as in Lemma~\ref{C}. If both $u\equiv 0$ and $v\equiv 0$, then the limit $\psi\in\Wqd$ of $\psi_j:=v/\|v_j\|_{\Wq}$ satisfies $\partial_a\psi-\Delta_D\psi=0$ with $\psi(0)=\xi\Psi$ leading to the contradiction $1=\xi r(H_{[0]})=\xi$ due to \eqref{1mio}. If $v\equiv 0$ but $u\not\equiv 0$, then $u\in\Wqd$ satisfies $\partial_au-\Delta_Du=-\alpha_1u^2$ and $u(0)=\eta U$ and thus $u=u_\eta$ with necessarily $\eta>1$ by the uniqueness statement of Theorem~\ref{A1}. Hence $\psi\in\Wqd$ satisfies  $\partial_a\psi-\Delta_D\psi=\beta_2u_\eta\psi$ with $\psi(0)=\xi\Psi$ giving the contradiction $\eta=\eta_1$ by \eqref{eta1}. If $u\equiv 0$ but $v\not\equiv 0$, then $v\in\Wqd$ satisfies $\partial_av-\Delta_Dv=-\beta_1v^2$ and $v(0)=\xi V$ what is impossible according to Theorem~\ref{A1} since $\xi<1$. Therefore, neither $u\equiv0$ nor $v\equiv 0$ and we conclude $\eta=\eta_*$. This proves the claim.
\end{proof}

\begin{lem}
There is no other bifurcation point to positive coexistence states on $\mathfrak{B}_1$.
\end{lem}

\begin{proof}
The assumption $(\eta,u_\eta,0)\in\mathfrak{B}_1$ being a bifurcation point to positive coexistence states corresponds to the case $\eta>1$, $u\not\equiv 0$, and $v\equiv 0$ in the proof of Lemma~\ref{S} and analogously implies $\eta=\eta_1$.
\end{proof}

This completes the proof of Theorem~\ref{T22}.

\section{Proof of Theorem~\ref{T222}}\label{sec 4}

Again, the main part of the proof of Theorem~\ref{T222} is a straightforward modification of Section~\ref{sec 2}, and we thus omit details. Let $\eta>1$ be fixed. Linearization around $(\xi,u_\eta,0)\in\mathfrak{T}_2$ entails the existence of a continuum $\mathfrak{R}^+$ in $\R\times\Wq\times\Wq$ of solutions to \eqref{7}-\eqref{6a}  bifurcating from $(\xi_0,u_{\eta},0)$, where $\xi_0:=\xi_0(\eta)\in (0,1)$ is given by
\bqn\label{xi0}
\xi_0:=\frac{1}{r(H_{[-\beta_2 u_{\eta}]})}\ .
\eqn
Near the bifurcation point $(\xi_0,u_{\eta},0)$, $\mathfrak{R}^+$ is a continuous curve in $\R^+\times\Wqd\times\Wqd$ and it can be shown exactly as in Lemma~\ref{C} or Lemma~\ref{S} that $\mathfrak{R}:=\mathfrak{R}^+\cap (\R^+\times\Wq^+\times\Wq^+)$ satisfies the alternatives:
\begin{itemize}
\item[(a)] $\mathfrak{R}\setminus\{(\xi_0,u_{\eta},0)\}$ is unbounded in $\R^+\times \Wqd\times \Wqd$, or
\item[(b)] $\mathfrak{R}$ reaches the boundary of $\R^+\times \Wqd\times \Wqd$ at some point $(\xi,u,v)\not=(\xi_0,u_{\eta},0)$ with $u\equiv 0$ or $v\equiv 0$.
\end{itemize}
If (b) occurs, choose a sequence $(\xi_j,u_j,v_j)$ in $\mathfrak{R}\cap (\R^+\times \Wqd\times \Wqd)$ converging toward $(\xi,u,v)\not=(\xi_0,u_{\eta},0)$. Putting $\phi_j:=u_j/\|u_j\|_{\Wq}$ and $\psi_j:=v_j/\|v_j\|_{\Wq}$, we may assume that $\phi_j\rightarrow\phi$ and $\psi_j\rightarrow\psi$ in $\Wq$. If  both $u\equiv 0$ and $v\equiv 0$, then $\phi\in\Wqd$ satisfies $\partial_a\phi-\Delta_D\phi=0$ with $\phi(0)=\eta\Phi$ and so $1=\eta r(H_{[0]})$ contradicting \eqref{1mio} and $\eta>1$. If $u\not\equiv 0$ but $v\equiv 0$, then $u=u_\eta$ according to Theorem~\ref{A1}, and $\psi\in\Wqd$ satisfies $\partial_a\psi-\Delta_D\psi=\beta_2u_\eta\psi$ with $\psi(0)=\xi\Psi$. Hence $\xi=\xi_0$ by \eqref{xi0} what is impossible since $(\xi,u,v)\not=(\xi_0,u_{\eta},0)$. Therefore, the only possibility is $u\equiv 0$ but $v\not\equiv 0$. In this case necessarily $\xi>1$ and $v=v_\xi$ in view of Theorem~\ref{A1}. So $\mathfrak{R}$ joins up with $\mathfrak{T}_1$ at $(\xi,0,v_\xi)$. We remark that then the relation
\bqn\label{xi1}
\eta r(G_{\xi})=1
\eqn
with $G_{\xi}$ given in \eqref{Z111} must hold, since $\phi\in\Wqd$ satisfies \eqref{Z0}. That no other bifurcation point(s) on $\mathfrak{T}_2$ or $\mathfrak{T}_1$ exist(s) is immediate by the previous observations. This yields Theorem~\ref{T222}.

\begin{rem}\label{R1}

Since the operator $A_\xi$ in \eqref{Axi} does not yield a suitable maximum principle, there is no analogue of Lemma~\ref{A2} for the spectral radius of $G_\xi$ and the only information we have on $r(G_{\xi})$ is that it is positive for each $\xi>1$ as observed in Section~\ref{sec 2.2}. Consequently, given $\eta>0$, we cannot decide {\it a priori} whether \eqref{xi1} holds for some $\xi>1$. However, for $\eta>1$, if $\mathfrak{R}$ joins up with $\mathfrak{T}_1$, then \eqref{xi1} must occur and the connection point $(\xi,0,v_\xi)$ on $\mathfrak{T}_1$ is determined by this relation. Then again, the relation \eqref{xi1} is also a sufficient condition for the existence of a continuous curve of positive coexistence solutions bifurcating from $\mathfrak{T}_1$. 

The same difficulty arises when considering bifurcation from $\mathfrak{T}_1$ with respect to $\xi$ when $\eta<1$ is fixed. In this case, \eqref{xi1} is again a necessary and sufficient condition for the existence of a bifurcation point  on $\mathfrak{T}_1$ to a curve of positive coexistence states, which then extends to an unbounded continuum.

\end{rem}

\end{document}